\documentclass{article}

\usepackage[english]{babel}

\usepackage[letterpaper,top=2cm,bottom=2cm,left=3cm,right=3cm,marginparwidth=1.75cm]{geometry}

\usepackage{verbatim}
\usepackage{caption}
\usepackage{float}
\usepackage{pifont}
\usepackage{enumerate}
\usepackage{amsmath}
\usepackage{graphicx}
\usepackage[colorlinks=true, allcolors=blue]{hyperref}
\newtheorem{theorem}{Theorem}[section]
\newtheorem{conjecture}[theorem]{Conjecture}
\newtheorem{lemma}[theorem]{Lemma}

\newtheorem{claim}{Claim}[section]
\newtheorem{subclaim}{Subclaim}[claim]

\def\qed{\hfill \rule{4pt}{7pt}}

\def\pf{\noindent {\it Proof. }} 

\title{Graphs with girth 8 and without longer even holes are 3-colorable}
\author{Yan Wang$^{1,}$ \thanks{Supported by National Key R\&D Program of China under grant No. 2022YFA1006400, National Natural Science Foundation of China under grant No. 12571376, email: yan.w@sjtu.edu.cn},  \;\; Rong Wu$^{2,}$ \thanks{Supported by National Key R\&D Program of China under Grant No. 2022YFA1006400, email: wurong@shmtu.edu.cn} \\\\
    \small $^1$School of Mathematical Sciences, CMA-Shanghai\\
    \small Shanghai Jiao Tong University, 800 Dongchuan Road, Shanghai 200240, China\\
    \small $^2$School of  Science \\
    \small Shanghai Maritime University,  Shanghai 201306, China}

\begin{document}
\maketitle

\begin{abstract}
For an integer $\ell\geq 2$, let ${\cal{H}}_{\ell}$ denote the family of graphs which have girth $2\ell$ and have no even hole of length greater than $2\ell$. Wu, Xu and Xu conjectured that every graph in $\bigcup_{\ell\geq 2} {\cal{H}}_{\ell}$ is $3$-colorable. Chen showed that every graph in $\bigcup_{\ell\geq 5} {\cal{H}}_{\ell}$ is $3$-colorable.
In this paper, we prove that every graph in ${\cal{H}}_4$ is $3$-colorable.
\end{abstract}

\section{Introduction}

All graphs considered in this paper are finite, simple, and undirected. 
Let $G$ be a graph and let $S$ be a subset of $V(G)$.
We use $G[S]$ to denote the subgraph of $G$ induced by $S$.
Let $x\in V(G)$ (we also write $x\in G$ if there is no confusion), we use $N_S(x)$ to denote the neighbours of $x$ in $S$.
%For subgraphs $H$ and $H'$ of $G$, we use $H \triangle H'$ to denote the symmetric difference of $H$ and $H'$, that is, 
%$V(H\triangle H')=V(H)\cup V(H')\backslash \{V(H)\cap V(H')\}$ and $E(H\triangle H')=E(H)\cup E(H')\backslash \{E(H)\cap E(H')\}$.
Let $P$ be an $(x, y)$-path, that is, the ends of $P$ are $x$ and $y$.
And we usually use $xPy$ to denote $P$, if there is no confusion, we just write it as $P$.
Let $P^*$ denote the set of internal vertices of $P$.
%Let $C$ be a cycle and $u, v$ be two vertices of $C$. We use $C(u, v)$ to denote the subpath of $C$ from $u$ to $v$ in clockwise order, and $C^*(u, v)$ to denote the set of internal vertices of $C(u, v)$.

A graph $G$ is $k$-colorable if there exists a mapping $c: V(G)\rightarrow \{1, 2, \cdots, k\}$ such that $c(u)\neq c(v)$ whenever $uv\in E(G)$.
The \textit{chromatic number $\chi(G)$} of $G$ is the minimum integer $k$ such that $G$ is $k$-colorable.
The clique number $\omega(G)$ of $G$ is the maximum integer $k$ such that $G$ contains a complete graph of size $k$.
For a graph $G$, if $\chi(G)=\omega(G)$, then we call $G$ a \textit{perfect} graph.
For a graph $H$, we say that $G$ is \textit{$H$-free} if $G$ induces no $H$ (i.e., $G$ has no induced subgraph isomorphic to $H$).
Let $\cal F$ be a family of graphs. We say that $G$ is \textit{$\cal F$-free} if $G$ induces no member of $\cal F$.
If there exists a function $\phi$ such that $\chi(G)\leq \phi(\omega(G))$ for each $G\in {\cal F}$, then we say that $\cal{F}$ is $\chi$-\textit{bounded class}, and call $\phi$ a \textit{binding function} of $\cal{F}$.
The concept of $\chi$-boundedness was posed by Gy\'{a}rf\'{a}s in 1975 \cite{g2}.
Studying which family of graphs can be $\chi$-bounded, and finding the optimal binding function for $\chi$-bounded class are important problems in this area.
Since clique number is a trivial lower bound of chromatic number, if a family of $\chi$-bounded graphs has a linear binding function, then this function must be the asymptotically optimal binding function of this family.
For recent progress on $\chi$-bounded problems, see \cite{ss}.

A hole in a graph is an induced cycle of length at least $4$.
A hole is said to be \textit{odd} (resp. \textit{even}) if it has odd (resp. even) length.
Addario-Berry, Chudnovsky, Havet, Reed and Seymour \cite{achrs}, and Chudnovsky and Seymour \cite{cs08}, proved that every even hole free graph has a vertex whose neighbours are the union of two cliques, which implies that $\chi(G)\leq 2\omega(G)-1$.
However, the situation becomes much more complicated for odd hole free graphs.
The Strong Perfect Graph Theorem \cite{cs02} asserts that a graph is perfect if and only if it induces neither odd holes nor their complements.
Confirming a conjecture of Gy\'{a}rf\'{a}s \cite{g2},
Scott and Seymour \cite{sso} proved that odd hole free graphs are $\chi$-bounded with binding function $\frac{2^{2^{\omega(G)+2}}}{48(\omega(G)+2)}$.
Ho\`{a}ng and McDiarmid \cite{hm} conjectured for an odd hole free graph $G$, $\chi(G)\leq 2^{\omega(G)-1}$.

%Sivaraman \cite{sivaraman} conjectured that $\chi(G)\leq \omega^2(G)$ for all short-holed graphs whereas the best known upper-bound is $\chi(G)\leq 10^{20}2^{\omega^2(G)}$ due to Scott and Seymour \cite{ss}.

A related direction concerns graphs with restricted hole lengths.
A graph is $\ell$-holed if every hole has length exactly $\ell$. 
Cook, Horsfield, Preissmann, Robin and Seymour \cite{chprsstv} gave a complete structural description of $\ell$-holed graphs for every $\ell\ge 7$.
The chromatic behavior of $\ell$-holed graphs strongly depends on the parity of $\ell$.
If $\ell\ge 6$ is even, then every $\ell$-holed graph is perfect, and hence $\chi(G)=\omega(G)$. 
In contrast, for odd $\ell\ge 7$, Wang and Wu \cite{ww25+} proved that
$\chi(G)\leq {\lceil {\ell \over {\ell-1}}\omega(G) \rceil}$,
and this bound is optimal.
A graph is said to be \textit{short-holed} if every hole of it has length 4. 
%A closely related notion is that of short-holed graphs, in which every hole has length exactly four. 
Sivaraman \cite{sivaraman} conjectured that every short-holed graph satisfies $\chi(G)\le \omega(G)^2$, while the best known general upper bound remains exponential, namely
$\chi(G)\leq 10^{20}2^{\omega(G)^2}$, due to Scott and Seymour \cite{ss}.

The girth of a graph $G$, denoted by $g(G)$, is the minimum length of a cycle in $G$.
Let $\ell\geq 2$ be an integer. Let ${\cal G}_{\ell}$ denote the family of graphs that have girth $2\ell+1$ and have no odd holes of length at least $2\ell+3$. 
The graphs in ${\cal G}_2$ are called \textit{pentagraphs}, and the graphs in ${\cal G}_3$ are called \textit{heptagraphs}.
Robertson \cite{NPRZ11} conjectured  that the Petersen graph is the only non-bipartite pentagraph which is $3$-connected and internally $4$-connected. Plummer and Zha \cite{PZ14} presented counterexamples to Robertson's conjecture, and conjectured that every pentagraph is $3$-colorable.
Xu, Yu and Zha \cite{xyz17} proved that every pentagraph is $4$-colorable.
Generalizing the result of \cite{xyz17}, Wu, Xu and Xu \cite{wxx22} proved that graphs in $\bigcup_{\ell\geq 2}{\cal G}_{\ell}$ are $4$-colorable and proposed the following conjecture.
\begin{conjecture} \label{wxx22}{\em{\cite{wxx22}}}
Graphs in $\bigcup_{\ell\geq 2}{\cal G}_{\ell}$ are $3$-colorable.
\end{conjecture}

Recently, Chudnovsky and Seymour \cite{cs22} confirmed that pentagraphs are $3$-colorable.
Wu, Xu and Xu \cite{wxx22+} showed that heptagraphs are $3$-colorable.
Chen \cite{c23} proved that all graphs in $\bigcup_{\ell\geq 5}{\cal G}_{\ell}$ are $3$-colorable.
Later, Wang and Wu \cite{ww24} confirmed Conjecture \ref{wxx22} holds for $\ell=4$, thereby completely resolving the Conjecture \ref{wxx22}.

Let ${\cal H}_{\ell}$ denote the family of graphs that have girth $2{\ell}$ and have no even holes of length at least $2\ell+2$. 
%%The case of odd girth  becomes increasingly difficult as $\ell$ grows. Indeed, longer shortest odd holes allow more parity-compatible induced paths, jumps, and induced $K_4$-subdivisions. So the exclusion of longer odd holes imposes  weaker local restrictions.
It is natural to study the chromatic number of every graph in $\mathcal{H}_{\ell}$. 
Similar to Conjecture \ref{wxx22}, Wu, Xu and Xu proposed the following (see Chen \cite{c25}). 
\begin{conjecture} \label{wxx}
Graphs in $\bigcup_{\ell\geq 2}{\cal H}_{\ell}$ are $3$-colorable.
\end{conjecture}

More recently,  Chen \cite{c25} proved that all graphs in $\bigcup_{\ell\geq 5}{\cal H}_{\ell}$ are $3$-colorable.
In contrast to the case of odd girth, the even girth problem becomes more difficult for small $\ell$.
When $\ell$ is large, the girth condition and the exclusion of longer even holes impose stronger distance constraints on induced theta graphs and induced 
$K_4$-subdivisions. 
However, for small even girth, extremal configurations such as short jumps and tight theta subgraphs may exist simultaneously without immediately creating either a shorter cycle or a forbidden longer even hole, leaving much less structural information.
In this paper, we prove that Conjecture \ref{wxx} holds for $\ell=4$.
\begin{theorem} \label{main theorem}
Graphs in ${\cal H}_4$ are $3$-colorable.
\end{theorem}
In fact, we prove the following structural characterization, which implies Theorem \ref{main theorem}.
\begin{theorem} \label{theorem 1}
All graphs in ${\cal H}_4$ have either a degree-$2$ vertex or a $K_1$-cut or $K_2$-cut.
\end{theorem}
\noindent\textit{Proof of Theorem \ref{main theorem}.} 
Suppose that Theorem \ref{main theorem} is false, and let $G\in\mathcal H_\ell$ be a minimal counterexample. 
Then  $d(G)\ge 3$. 
By Theorem \ref{theorem 1}, $G$ has a $K_1$-cut or a $K_2$-cut. Consequently, $G$ can be decomposed into two induced subgraphs $G_1$ and $G_2$ whose intersection is either a vertex or an edge.
By the minimality of $G$, both $G_1$ and $G_2$ are 3-colorable, and therefore their colorings can be combined to obtain a 3-coloring of $G$, contradicting the assumption that $G$ is a counterexample.
\qed

The organization of this paper is as follows. 
In Section 2, we present  some notations and preliminary lemmas.
In Section 3, we  complete the proof of Theorem \ref{theorem 1}.

\section{Preliminary}

In this section, we collect some useful lemmas. 
The author of \cite{c25} proved the following lemma .

Let $G$ be a graph.
We say that a path  $P$ of $G$ is an ear of $G$ if all vertices in $P^*$ have degree-$2$ in $G$ while both ends of $P$ have degree at least $3$. 
A theta graph is a graph that consists of a pair of distinct vertices joined by three internally disjoint paths.
%Let $C$ be a hole of a graph $G$.
%A path $P$ of $G$ is a chordal path of $C$ if $C\cup P$ is an induced theta-subgraph of $G$.

\begin{lemma} \label{lemma 0}{\em \cite{c25}}
Let $\ell\geq 2$ be an integer and $H$ be an induced theta subgraph of a graph $G\in {\cal H}_{\ell}$.
\begin{enumerate}[(1)]
\item When all cycles in $H$ are even, either all ears of $H$ have length $\ell$ or exactly one ear of $H$ has length one.
\item When some cycle in $H$ is odd, the length of the ear of $H$ shared by its odd cycles is one or larger than $\max\{\ell, |P|, |P'|\}$, where $P$, $P'$ are the ears of $H$ shared by cycles having different parity.
\end{enumerate}
\end{lemma}
From the above lemma, it can be inferred that if all three paths of the theta graph are at least of length $2$, then this theta graph induces three even holes or two odd holes and one even hole.

Let $C$ be an even hole of a graph $G$ and 
let $P$ be an induced $(s, t)$-path such that $V(C)\cap V(P^*)=\emptyset$, $s, t\in V(C)$.
Then we call $P$ a jump or an $(s, t)$-jump over $C$.
When $s, t$ are adjacent and $C\backslash \{s, t\}$ is anticomplete to $P^*$, 
we refer to such $P$ as $(s, t)$-link over $C$.
%(we also write $(s, t)$-link if there is no confusion).
When $s, t$  are not adjacent,
let $Q_1, Q_2$ be the internally disjoint $(s, t)$-paths of $C$.
When no vertex in $V(Q^*_1\cup Q^*_2)$ has  a neighbor in $V(P^*)$, we say that $P$ is a short jump over $C$. 
If some vertex in $V(Q^*_1)$ has a neighbor in $V(P^*)$ and no vertex in $V(Q^*_2)$ has a neighbor in $V(P^*)$, we say that $P$ is a local jump over $C$ across $Q^*_1$.
In particular, when $|V(Q^*_1)|=1$, we say that $P$ is a local jump over $C$ across one vertex.
For the sake of clarity, below we will denote the local $(s, t)$-jump across $c$ as $s-c-t$ local jump.
When $|V(Q^*_1)| \neq 1$, $|V(Q^*_2)| \neq 1$, 
and the only neighbor of $s$ in $V(Q_1)$ is denoted by $c$, which has a neighbor in $V(P^*)$, and no vertex in $V(Q^*_2)$ has a neighbor in $V(P^*)$, then we denote such $P$ as an $s-c-t$ link over $C$.
If no confusion arises, we simply omit ``over $C$''.

\begin{lemma}\label{lemma 1}  {\em \cite{c25}}
Let $C$ be a hole of a graph $G$ such that no vertex outside $C$ has two neighbors in $C$.
Assume that $G$ has neither a degree-$2$ vertex nor $K_1$-cut nor $K_2$-cut.
For any $3$-vertex path $xyz$ on $C$, one of the following holds.
\begin{enumerate}[(1)]
    \item Either $x$ or $z$ is one end of a short jump over $C$.
    \item There is a local $(x, z)$-jump over $C$ across $y$.
    \item The vertex $y$ is an end of a short jump or a local jump over $C$ across one vertex.
\end{enumerate}
\end{lemma}

Let $P$ be a short $(s, t)$-jump over $C$, and $Q_1, Q_2$ be the $(s, t)$-paths of $C$. Note that $C\cup P$ is an induced theta subgraph of $G$  whose ears at least $2$, and both $PQ_1$ and $PQ_2$ have the same parity as $C$ is even. When $PQ_1$ is odd, we say that $P$ is of \textit{type-o};
and when $PQ_1$ is even, we say that $P$ is of \textit{type-e}.
By Lemma \ref{lemma 0} (2), we have $|P|>\ell$ when $P$ is of type-o.
When $P$ is of type-e, it follows from Lemma \ref{lemma 0} (1) that $|P|=|Q_1|=|Q_2|=\ell$.

Let $C$ be an even hole and $P_i$ be an $(u_i, v_i)$-jump over $C$ for each integer $1\leq i\leq 2$.
If $u_1,v_1,u_2,v_2$ are distinct vertices and $u_1,u_2,v_1,v_2$ appear on $C$ in this order, then we say that $P_1, P_2$ are \textit{crossing}; Otherwise , they are \textit{parallel}. 
Note that by our definition, $P_1, P_2$ are parallel when they share an end. 

\begin{lemma}\label{c4.1}  {\em \cite{c25}}
Let $\ell\geq 3$ be an integer and $C$ be an even hole of a graph $G\in {\cal H}_\ell$.
Then all short jumps over $C$ of type-e have the same ends.
\end{lemma}

\begin{lemma}\label{lemma 2}  {\em \cite{c25}}
Let $l\geq 4$ be an integer and $C$ be an even hole of a graph $G\in {\cal H}_\ell$.
For any integer $1\leq i\leq 2$, let $P_i$ be a short  $(u_i, v_i)$-jump over $C$ such that $P_1$ and $P_2$ are of type-o when they are parallel.
If $G$ has neither a degree-$2$ vertex nor a $K_1$-cut nor $K_2$-cut,
then one of the following holds.
\begin{enumerate}[(1)]
    \item 
    When $u_1=u_2$, we have $v_1v_2\in E(G)$.
    \item 
    When $|\{u_1, u_2, v_1, v_2\}|=4$, either $u_1u_2, v_1v_2\in E(G)$ or
    $u_iu_{3-i}v_i$ is a path on $C$ for some integer $1\leq i\leq 2$.
\end{enumerate}
\end{lemma}
From this, it can be seen that when $P$ is of type-o and $Q$ is of type-e, if they  cross each other, they satisfy Lemma \ref{lemma 2} (2).
However, if they are parallel, they do not necessarily satisfy Lemma \ref{lemma 2}.

\begin{lemma}\label{lemma 3}  {\em \cite{c25}}
Let $l\geq 4$ be an integer and $C$ be an even hole of a graph $G\in {\cal H}_\ell$.
Let $P_1$ be a local $(u_1, v_1)$-jump over $C$ across one vertex $s_1$.
Let $P_2$ be a short or local $(u_2, v_2)$-jump over $C$.
Assume that $P_1$, $P_2$ are parallel and $G$ has neither a degree-$2$ vertex nor a $K_1$-cut nor $K_2$-cut.
Then one of the following holds.
\begin{enumerate}[(1)]
    \item 
    When $P_2$ is local and across one vertex $s_2$, if $P_1$, $P_2$ do not share an end, then there is a short $(s_i, w_j)$-jump over $C$,
    where $\{i,j\}=\{1,2\}$ and $w_j\in \{u_j, v_j, s_j\}$.
    \item 
    When $P_2$ is short and of type-o, $P_1$ and $P_2$ share an end.
\end{enumerate}
\end{lemma}

\begin{lemma}\label{lemma 4}
Let $\ell\geq 4$ be an integer and $C$ be an even hole of a graph in ${\cal H}_{\ell}$.
Let $S$ be a set of ends of all short jumps over $C$.
Assume that $G$ has neither a degree-$2$ vertex nor a $K_1$-cut or $K_2$-cut and some short jump over $C$ is of type-o.
Then $V(Q^*)\cap S\neq \emptyset$ for any $5$-vertex path $Q$ on $C$.
\end{lemma}
\pf Suppose not.
Set $C=v_1v_2\cdots v_8v_1$.
Let $Q=v_1v_2v_3v_4v_5$ and $\{v_2, v_3, v_4\}\cap S=\emptyset$.
By Lemma \ref{lemma 3} (2), there does not exist a $v_2-v_3-v_4$ local jump.
And by Lemma \ref{lemma 1}, either there is a $v_1-v_2-v_3$ local jump or  a $v_3-v_4-v_5$ local jump.
Let $P$ be a type-o short jump.
Suppose there is a $v_1-v_2-v_3$ local jump. Then, by Lemma \ref{lemma 3}, one end of $P$ is $v_1$.
If there is also a $v_3-v_4-v_5$ local jump simultaneously,
the other end of $P$ is $v_5$.
Therefore, the ends of all type-o short jumps are $v_1$ and $v_5$.
Moreover, if there exists a type-e short jump, then by Lemma \ref{c4.1}, 
the ends of any type-e short jumps are also $v_1$ and $v_5$.
That is $S=\{v_1, v_5\}$.
Now, $\{v_6, v_7, v_8\}\cap S=\emptyset$.
By Lemma \ref{lemma 1}, either there is a $v_1-v_8-v_7$ local jump, or there is a $v_5-v_6-v_7$ local jump.
If there is a $v_1-v_8-v_7$ local jump,
then by Lemma \ref{lemma 3}
 there is a short $(v_4, v)$-jump where $v\in \{v_1, v_8, v_7\}$,
or there is a short $(v_8, v)$-jump where $v\in \{v_3, v_4, v_5\}$, which
contradicts  $S=\{v_1, v_5\}$.
So, $v_1-v_2-v_3$ local jump exists, denoted as $Q'$, while $v_4-v_5-v_6$ local jump does not exist.
Now, one of the ends of all  short jumps of type-o is $v_1$.

Next, we prove that there is no  $(v_1,v)$-short jump of type-o where $v\in \{v_5, v_6, v_7\}$.
Suppose not.
Assume that there is a  $(v_1,v_7)$-short jump,  denoted by $P$.
then $|P|$ is odd.
Let $P$ and $Q'$ be chosen with $|P\cup Q'|$ as small as possible.
Then there exists an edge with one end in $V(P^*)$ and the other end in $V(Q'^*)$, 
otherwise $P\cup Q'\cup v_3v_4v_5v_6v_7$  induces an even hole of length greater than 8,
a contradiction.
Then $P\cup Q'$  induces a $(v_2, v_7)$-jump denoted as $P'$.
Since $v_2\notin S$, $v_1$ or $v_3$ has a neighbor in $P'$.
And $N(v_3)\cap P'\neq \emptyset$, otherwise  it contradicts 
the edge with one end in $V(P^*)$ and the other end in $V(Q'^*)$.
Then either there is  a  $(v_1, v_3)$-short jump or there is a  $(v_3, v_7)$-short jump, which would contradict the fact that $v_3\notin S$.
So, there is no  $(v_1,v_7)$-short jump.
%$P$.

Assume that there is a  $(v_1,v)$-short jump of type-o,
denoted by $P$ where $v\in \{v_5, v_6\}$ ,
then the parity of $|P|$ is different from that of the path $v_1v_2v_3\cdots v$.
Similarly, either $P\cup Q'\cup v_3v_4\cdots v$ induces an even hole of length greater than 8,
or there is a  $(v_2, v)$-short jump
or a  $(v_3, v)$-short jump
or a  $(v_1, v_3)$-short jump, 
a contradiction.
Then there is no  short jump of type-o, a contradiction.
\qed

%\begin{corollary}\label{cor}
%Let $C$ be an even hole and set $C=v_1v_2\cdots v_8v_1$.
%If there is a $v_1-v_2-v_3$ local jump and there are no short $(w, v)$-jumps of type-o where $w\in \{v_1, v_2, v_3\}$ and $v\in V(C)$ with $w\neq v, wv\notin E$,
%then there is no local jump across one vertex that share an end with $v_1$ or $v_3$.
%\end{corollary}

\begin{lemma}\label{lemma 5}
Let $C$ be an even hole. Set $C=v_1v_2\cdots v_7v_8v_1$.
And let $P$  be  a $(v_1, v_3)$-short jump, 
and let $Q$ be 
a $(v_2, v_6)$-short jump.
Then $P$ is anticomplete to $Q$.
\end{lemma}
\pf Suppose not.
Let $v\in V(Q^*)$ be closest to $v_6$ and have a neighbor on $P^*$.
Then $vv_2\in E$. 
Otherwise, there is a short $(v_1, v_6)$-jump or a short $(v_3, v_6)$-jump, which contradicts  Lemma \ref{lemma 2}.
Now, $vQv_6$ has length at least $3$.
If $v$ has a unique neighbor $v'$ on $P$, then $v_1P v_3v_4\cdots v_8v_1\cup v'vQv_6$ is a theta graph.
Since $P$ is a  $(v_1, v_3)$-short jump, $|P|$ is odd. 
Hence, $P\cup v_3v_4\cdots v_8v_1$ is an odd hole.
So either $v_1Pv'vQv_6v_7v_8v_1$ or $v'Pv_3v_4v_5v_6Qvv'$ is an even hole.
We may assume that $v_1Pv'vQv_6v_7v_8v_1$ is an even hole and of length  $8$, so $|v_1Pv'|=1$.
Now, $v_1v_2vv'v_1$ is a $4$-hole, which contradicts  $g(G)$.

Similarly, if $v$ has at least two neighbors  on $P$, 
let $u_1\in N(v)$ be closest to $v_1$ and $u_2\in N(v)$ be closest to $v_2$, then $u_1u_2\notin E$. 
Since $v_1Pu_1vu_2Pv_3v_4\cdots v_8v_1\cup vQv_6$ is a theta graph
and $v_1Pu_1vu_2Pv_3v_4\cdots v_8v_1$ has length greater than $8$, it is an odd hole.
So, either $v_1Pu_1vQv_6v_7v_8v_1$ or $u_2Pv_3v_4v_5v_6Qvu_2$ is an even hole.
We may assume that $v_1Pu_1vQv_6v_7v_8v_1$ is an even hole and has length $8$, so $|v_1Pu_1|=1$.
Now, $v_1v_2vu_1v_1$ is a $4$-hole, which contradicts  $g(G)$.
\qed

\section{Proof of Theorem~\ref{theorem 1}}

Assume for contradiction that the statement is false.
Let $C$ be an even hole of a graph $G\in {\cal H}_4$ and write $C=v_1v_2\cdots v_8v_1$.
Let $S$ be the set consisting of ends of short jumps over $C$.

%\noindent \textbf{Claim 3.1} 
\begin{claim}
$S\neq \emptyset$.
\end{claim}
\pf Suppose it is false. 
That is, there is no short jump over $C$.
By Lemma \ref{lemma 3} (1),
every pair of local jumps  across one vertex either share an end or cross.
So there is a $5$-vertex path $Q$ on $C$ containing the ends of all local jumps across one vertex over $C$.
Then there exists a $3$-vertex path of $C\backslash V(Q)$ that contradicts  Lemma \ref{lemma 1}.
\qed

%\noindent \textbf{Claim 3.2} 
\begin{claim}
$|S|\neq 2$.
\end{claim}
\pf Suppose it is false.
If there exists a short jump of type-o, 
then we may assume that $v_1v_2v_3v_4v_5$ contains the ends of all short jumps.
Now, $\{v_6, v_7, v_8\}\cap S=\emptyset$, which contradicts  Lemma \ref{lemma 4}.
Hence, every short jump is of type-e.

Suppose that there exists a short $(v_1, v_5)$ jump of type-e, that is, $S=\{v_1, v_5\}$.
Since $S=\{v_1, v_5\}$, by Lemma \ref{lemma 1},
for $3$-vertex path $v_6-v_7-v_8$,
either there is a $v_7-v_8-v_1$ local jump, or there is a $v_6-v_7-v_8$ local jump, or there is a $v_5-v_6-v_7$ local jump.
Similarly, for $3$-vertex path $v_2-v_3-v_4$,
either there is a $v_1-v_2-v_3$ local jump, or there is a $v_2-v_3-v_4$ local jump, or there is a $v_3-v_4-v_5$ local jump.
Now, we show that there is no  $v_7-v_8-v_1$ local jump, or  $v_5-v_6-v_7$ local jump, or $v_1-v_2-v_3$ local jump, or $v_3-v_4-v_5$ local jump.
Suppose otherwise.
We may assume that there is a $v_7-v_8-v_1$ local jump, %then by Corollary \ref{cor},
then there is no $v_1-v_2-v_3$ local jump.
Otherwise, there exists a short jump with an end $v_7, v_8$, $v_2$, or $v_3$ or there exists an even hole containing $v_3v_4\cdots v_7$ of length at least $10$,
a contradiction.  
If there is a $v_2-v_3-v_4$ local jump, then by Lemma \ref {lemma 3}, there exists a short jump with an end $v_8$ or $v_3$;
and if there is a $v_3-v_4-v_5$ local jump, then by Lemma \ref {lemma 3}, there exists a short jump with an end $v_8$ or $v_4$.
A contradiction.
So, by Lemma \ref{lemma 1}, there exist a $v_6-v_7-v_8$ local jump and a $v_2-v_3-v_4$ local jump.
However, by Lemma \ref {lemma 3}, $v_3$ or $v_7$ is an end of some short jumps, which contradicts  $S=\{v_1, v_5\}$.
\qed

%\noindent \textbf{Claim 3.3} 
\begin{claim}
$|S|\neq 3$.
\end{claim}
\pf Suppose not.
Then there exist two short jumps sharing an end.
Clearly, these two short jumps are parallel and one of the short jumps is of type-o.
Assume that the shared end is $v_1$.
We may assume that one of the short jumps, say $P_1$, is a $(v_1, w_1)$-short jump, and the other short jump, say $P_2$, is a $(v_1, w_2)$-short jump where $w_1, w_2\in \{v_3, v_4, \cdots, v_7\}$ and $w_1\neq w_2$.
Then $S=\{v_1, w_1, w_2\}$.
By Lemma \ref{lemma 4}, $\{v_6, v_7\}\cap S\neq \emptyset$ and $\{v_3, v_4\}\cap S\neq \emptyset$.
Since $\{v_6, v_7\}$ is anticomplete to $\{v_3, v_4\}$, by Lemma \ref{lemma 2}, $P_1$ and $P_2$ cannot both be of type-o.
So we may assume that $P_1$ is of type-o and $P_2$ is of type-e, then
$w_2=v_5$, a contradiction.
\qed

\medskip
Therefore, we have $|S|\geq 4$.

%\noindent \textbf{Claim 3.4} 
\begin{claim}\label{claim 3.4}
For any $(s,t)$-short jump over $C$, $d_C(s,t)\neq 2$.
\end{claim}
\pf Suppose not. We may assume that there exists a $(v_8, v_2)$-short jump.

%\noindent \textbf{Claim 3.4.1} 
\begin{subclaim}
There does not exist a short jump crossing  $(v_8, v_2)$-short jump.
\end{subclaim}
\pf Suppose not. We divide into three cases.

\medskip
\noindent \textbf{Case 3.4.1} There is no $(v_1, v_3)$-short jump.

Suppose not. We assume that there is a $(v_1, v_3)$-short jump.
Then we show that there is a $(v_6, w)$-short jump or $(v_5, w)$-short jump where $w\in V(C)$.
Suppose not. By Lemma \ref{lemma 4},
there exists a $(v_7, w_1)$-short jump and a $(v_4, w_2)$-short jump where $w_1, w_2\in V(C)$.
Then either $(v_7, w_1)$-short jump or $(v_4, w_2)$-short jump is of type-o.
We may assume that the $(v_4, w_2)$-short jump is of type-o.
Since both $(v_8, v_2)$-short jump and $(v_4, w_2)$-short jump are of type-o, and $v_4v_2\notin E$, then $w_2=v_1$.
Now, there is a $(v_1, v_4)$-short jump.
However, $v_7v_1\notin E$ and $v_7v_4\notin E$, and
$(v_7, w_1)$-short jump contradicts  Lemma \ref{lemma 2}.
So, there is a $(v_6, w)$-short jump or a $(v_5, w)$-short jump such that $w\in V(C)$.
We assume that there is a $(v_6, w)$-short jump where $w\in V(C)$.
Since $v_6v_1\notin E$ and $v_6v_3\notin E$ , by Lemma \ref{lemma 2},
$w=v_2$ and since $v_6v_8\notin E$, $(v_6, v_2)$-short jump is of type-e.
Since $v_5$ is not adjacent to $v_8, v_1, v_2$ or $v_3$,
by Lemma \ref{lemma 2},
there cannot  exist a short jump of type-o with end $v_5$.
If there were a short jump with $v_5$ as an end,
it could be of type-e,
which would contradict  Lemma \ref{c4.1}.
So, there is no short jump with $v_5$ as an end.
Similarly, there is no short jump with $v_4$ as an end.

Since $d(v_4)\geq 3$, there exists a $(v_4, v_5)$-link or $(v_3, v_4)$-link.
If there exists a $(v_4, v_5)$-link $P'$,
since $G$ has no $K_2$-cut,
$V(G)\backslash \{v_4,v_5\}$ is connected.
Then there is a path $P$ where $P\cap P'^*\neq \emptyset$ such that one  end of $P$, denoted by $v$, is connected to $V(C)\backslash \{v_4, v_5\}$.
If $vv_6\in E$, then either there is a $v_4-v_5-v_6$ local jump, which contradicts  $(v_2, v_8)$-short jump by Lemma \ref{lemma 3};
or there is a $(v_4, v_6)$-short jump, 
which contradicts  $(v_2, v_8)$-short jump by Lemma \ref{lemma 2}.
If $v$ is adjacent to $v_7, v_8, v_1$ or $v_2$, then there exists a  short jump with end $v_5$ or $v_4$, a contradiction. 
So $vv_3\in E$,
then either there is a $v_3-v_4-v_5$ local jump, which contradicts  $(v_2, v_8)$-short jump by Lemma \ref{lemma 3};
or there is a $(v_3, v_5)$-short jump, 
which contradicts  $(v_2, v_8)$-short jump by Lemma \ref{lemma 2}.
So there does not exist a $(v_4, v_5)$-link.

Since $d(v_5)\geq 3$ and  no $(v_4, v_5)$-link exists, there exists a $(v_5, v_6)$-link.
Similarly, $V(G)\backslash \{v_5,v_6\}$ is connected.
So there is a $v_5-v_6-v_2$ link.
Otherwise, there is a $v_5-v_6-v_7$ local jump, which contradicts  $(v_2, v_8)$-short jump.
Since $d(v_4)\ge 3$, there is a $v_2-v_3-v_4$ local jump or a $v_1-v_3-v_4$ link.
If there exists a $v_2-v_3-v_4$ local jump $Q$, let $(v_2, v_8)$-short jump be $P$.
Choose $P$ and $Q$ such that $|P|+|Q|$ as small as possible.
Then $P\cup Q$ induces a $(v_8, v_3)$-short jump or $(v_8, v_4)$-short jump or an even hole with length greater than $8$, a contradiction.
So there does not exist a $v_2-v_3-v_4$ local jump
and there is a $v_1-v_3-v_4$ link.

Let $Q$ and $P$ denote the  $v_1-v_3-v_4$ link and $v_5-v_6-v_2$ link, respectively.
Let $v''_{4}\in N_Q(v_3)$  that is closest to $v_1$,
and let $v'_{5}\in N_P(v_6)$ that is closest to $v_2$.
Then, apart from $v_4v_5$, there are no other edges between $v_1Qv_4$ and $v'_{5}Pv_5$.
Otherwise, there is a $(w, v)$-short jump where $w\in \{v_1, v_3, v_4\}$ and $v\in \{v_5, v_6\}$,
but none of these can  exist.
If $P^*-V(v'_{5}Pv_5)$ is not anticomplete to $Q$, then there is a $(w, v_6)$-short jump where $w\in \{v_1, v_3, v_4\}$
or there is a $(v_2, v_4)$-short jump or a $v_2-v_3-v_4$ local jump, a contradiction.
Then $v_5Pv_2v_1Qv_4v_5$ is an even hole with girth greater than $8$, a contradiction.
Thus, there is no $(v_1, v_3)$-short jump.
\qed

\medskip
\noindent \textbf{Case 3.4.2} There is no $(v_1, v_4)$-short jump.

Suppose not. 
%Now we assume that there is a $(v_1, v_4)$-short jump.
Since $v_6, v_7$ are not adjacent to $v_1$ or $v_4$,
by Lemma \ref{lemma 2}, there is no short jump with $v_6$ or $v_7$ as an end.
By Lemma \ref{lemma 4}, there must be a short jump with $v_5$ as the end.
And by Lemma \ref{lemma 2}, it can only be the $(v_5, v_1)$-short jump.
Then we show that there is no $(v_6, v_7)$-link.
Suppose not.
There is a $v_6-v_7-v_8$ local jump or $v_5-v_6-v_7$ local jump, which contradicts  $(v_1, v_4)$-short jump by Lemma \ref{lemma 3}.
Since $d(v_6)\geq 3$, there is a $v_4-v_5-v_6$ local jump or $v_6-v_5-v_1$
link.
If there is a $v_4-v_5-v_6$ local jump, which contradicts  $(v_2, v_8)$-short jump by Lemma \ref{lemma 3}.
Thus, there exists a $v_6-v_5-v_1$ link.
Similarly, there is a $v_7-v_8-v_1$ local jump or $v_7-v_8-v_2$ link.
If there is a $v_7-v_8-v_1$ local jump,
it is denoted as $P$.
Meanwhile, $(v_1, v_4)$-short jump is denoted as $Q$.
Since $G$ does not contain even holes with a length greater than $8$,
and by Lemma \ref{lemma 2}, there is no $(v_7, v_4)$-short jump.
Thus $P$ and $Q$ induce a $(v_4, v_8)$-short jump of type-e.
Therefore, either the vertex on $Q$ adjacent to $v_4$, denoted by $v$,
has a neighbor on $P$,
or the vertex on $Q$ at a distance of $2$ from $v_4$, denoted by $w$, has a neighbor on $P$.
If $|N_Q(v)|=1$, since there is no $(v_4, v_7)$-short jump, and $P\cup \{v, v_2, v_3, \cdots, v_6\}$ is a theta graph and $(v_1, v_4)$-short jump is of type-o, then $P\cup C\backslash \{v_8\}$ contains an even hole of length greater than $8$, a contradiction.
%Then, there is a $(v_1, v_4)$-short jump of length $5$, 
%which contradicts to $(v_1, v_4)$-short jump of type-o.
%and its union with $v_4v_5\cdots v_8v_1$ yields an even hole of length $10$, a contradiction.
If $|N_Q(v)|>2$, $\{v_8, v\}\cup P$ contains an even hole of length $4$, a contradiction.
So $|N_Q(v)|=0$.
If $|N_Q(w)|\ge 1$, similarly, $P\cup wQv_4\cup C\backslash \{v_8\}$ contains a theta graph,  then there exists an even hole of length greater than $8$, a contradiction.
So, there is no $v_7-v_8-v_1$ local jump, only a $v_7-v_8-v_2$ link.
Moreover, there are no other edges between the $v_7-v_8-v_2$ link and the $v_6-v_5-v_1$ link except $v_1v_2$ and $v_6v_7$;
Otherwise, there is a $(v_7, v)$-short jump where $v\in \{v_1, v_5, v_6\}$
or  a $(v, w)$-short jump where $v\in \{v_2, v_8\}, w\in \{v_5, v_6\}$
or a $v_7-v_8-v_1$ local jump, a contradiction.
Then the union of $v_7-v_8-v_2$ link, $v_6-v_5-v_1$ link, $v_1v_2$ and $v_6v_7$ is an even hole of length greater than $8$, a contradiction.
So, there is no $(v_1, v_4)$-short jump.
\qed

\medskip

\noindent \textbf{Case 3.4.3} There is no $(v_1, v_5)$-short jump.

Suppose not.
%Now we assume that there is a $(v_1, v_5)$-short jump.
If there exists a short jump with end $v_7$, by Lemma \ref{lemma 2},
it can only be a $(v_7, v_5)$-short jump,
but $v_5v_2\notin E$, which contradicts  Lemma \ref{lemma 2} (2).
So, there does not exist a short jump  with end $v_7$.
Similarly, there does not exist a short jump  with end $v_3$.
Then either there is no short jump with end $v_6$ or there is no short jump with end $v_4$.
Otherwise, by Lemma \ref{lemma 2}, there exist a $(v_6, v_2)$-short jump
of type-e and a $(v_4, v_8)$-short jump of type-e,
%which contradicts to $(v_2, v_8)$-short jump since $v_6v_8\notin E, v_2v_4\notin E$.
but now $v_2v_8\notin E, v_4v_6\notin E$,
which contradicts  Lemma \ref{lemma 2}.
So we may assume that there is no short jump with end $v_4$.

Now we show that there is no $(v_3, v_4)$-link.
Suppose not. Since $G\backslash \{v_3, v_4\}$ is connected,
either there is a $v_3-v_4-v_5$ local jump or there is a $v_2-v_3-v_4$ local jump.
If there is a $v_3-v_4-v_5$ local jump, by Lemma \ref{lemma 3}, which contradicts  $(v_2, v_8)$-local jump.
If there is a $v_2-v_3-v_4$ local jump, and denote it by $P$,
and denote the $(v_2, v_8)$-short jump by $Q$.
Then either $P\cup Q$ together with the path $v_4v_5v_6v_7v_8$ contains an even hole of length greater than $8$;
or $P\cup Q$ contains a $(v_3, v_8)$-short jump, which would contradict the $(v_1, v_5)$-short jump;
or $P\cup Q$ contains a $(v_4, v_8)$-short jump of type-e, a contradiction.
%in which case there must be a $(v_2, v_4)$-short jump that contradicts the $(v_2, v_8)$-short jump.
Thus there is  no $(v_3, v_4)$-link.

Since $d(v_4)\ge 3$, if there is a $v_4-v_5-v_6$ local jump,
by Lemma \ref{lemma 3}, which contradicts  $(v_2, v_8)$-local jump;
so there is a $v_4-v_5-v_1$ link, denoted by $P$.
Similarly, either there is a $v_3-v_2-v_1$ local jump, or there is a $v_3-v_2-v_8$ link, or there is a $v_3-v_2-v_6$ link.
If there is a $v_3-v_2-v_1$ local jump, denoted by $Q$,
choose $P$ and $Q$ such that $|P\cup Q|$ is minimized.
Then $P$ and $Q$ share at most one edge, this edge is incident to $v_1$, and $P^*$ is anticomplete to $Q^*$ by Lemma \ref{lemma 2} and there is not  $(v_3, v_4)$-link.
Then $P\cup Q\cup v_3v_4$ contains an even hole of length greater than $8$, a contradiction.
If there is a $v_3-v_2-v_6$ link, denoted by $Q$,
similarly, $P^*$ is anticomplete to $Q^*$.
Then $P\cup Q\cup v_3v_4\cup v_6v_7v_8v_1$ induces an even hole of length greater than $8$, a contradiction.
If there a $v_3-v_2-v_8$ local jump, denoted by $Q$,
by Lemmas \ref{lemma 2} and \ref{lemma 5},
$P^*$ is anticomplete to $Q^*$.
Then $P\cup Q\cup v_3v_4\cup v_8v_1$ induces an even hole of length greater than $8$, a contradiction.
So there is no $(v_1, v_5)$-short jump.\qed

Since $v_6, v_5$ are not adjacent to $v_8, v_2$, and there does not exist a short jump crossing  $(v_2, v_8)$-short jump,
therefore, by Lemma \ref{lemma 2},
there is no short jump with $v_6, v_5$ as the end.

%\noindent \textbf{Claim 3.4.2} 
\begin{subclaim}
There is no short jump parallel to $(v_2, v_8)$-short jump.
\end{subclaim}
\pf Suppose not. 
By Lemma \ref{lemma 4},
there exist short jumps with $v_7$ and $v_4$ as ends, respectively. Furthermore, by Lemma \ref{lemma 2}, the short jump with $v_4$ as an end is the $(v_4, v_8)$-short jump of type-e, and the short jump with $v_7$ as an end is the $(v_3, v_7)$-short jump.
Then we show that there is no $(v_5, v_6)$-link.
Otherwise, either there is a $v_6-v_5-v_4$ local jump or a $v_5-v_6-v_7$ local jump.
By Lemma \ref{lemma 3}, both contradict the $(v_2, v_8)$-short jump.
So there is no $(v_5, v_6)$-link.

Since $d(v_5)\ge 3$ and $d(v_6)\ge 3$, there exists a $(v_4, v_5)$-link and a $(v_6, v_7)$-link. When there is a $(v_4, v_5)$-link, since $\{v_4, v_5\}$ is not a cut, there is either a $v_3-v_4-v_5$ local jump or a $v_5-v_4-v_8$ link. Moreover, a $v_3-v_4-v_5$ local jump cannot exist; Otherwise  it would contradict the $(v_2, v_8)$-short jump. Therefore, only a  $v_5-v_4-v_8$ link is possible.
When there is a $(v_6, v_7)$-link, there is either a $v_6-v_7-v_3$ link or a $v_6-v_7-v_8$ local jump. If there is a $v_6-v_7-v_8$ local jump, 
%it contradicts Corollary \ref{cor}, 
together with $(v_2, v_8)$-short jump and $v_2v_3\cdots v_6$, either there is a short jump with end $v_6$, or there exists a $(v_2, v_7)$-short jump or there exists an even hole with length greater than $8$, a contradiction.
So only the $v_6-v_7-v_3$ link can exist.
Denote the $v_6-v_7-v_3$ link by $Q$, the $v_5-v_4-v_8$ link by $P$, and let $v'$ be the neighbor of $v_3$ on $Q$, and $v$ be the neighbor of $v_8$ on $P$. Then if there is an edge between $P^*$ and $Q^*$, it can only be $vv'$. Then $v_8vv'v_3v_2v_1v_8$ is a hole of length $6$, a contradiction. Hence $P$ and $Q$ are anticomplete. In this case, $P\cup Q \cup v_5v_6\cup v_8v_1v_2v_3$ forms an even hole of length greater than $8$, a contradiction. In conclusion, there is no short jump parallel to the $(v_2, v_8)$-short jump.
\qed

\medskip

%\noindent \textbf{Claim 3.5} 
\begin{claim}\label{claim 3.5}
For any $(s, t)$-short jump over $C$, $d_C(s, t)\neq 3$.
\end{claim}
\pf Suppose not. We may assume that there is a $(v_1, v_6)$-short jump.
Since $\{v_3, v_4\}$ is anticomplete to $\{v_1, v_6\}$, by Lemma \ref{lemma 2} there is no short jump with end $v_3$ or $v_4$. 
By Lemma \ref{lemma 4}, $\{v_2, v_5\}\subseteq S$.
Then we prove that there does not exist a short jump crossing  $(v_1, v_6)$-short jump.
Suppose not. By Lemma \ref{lemma 2} and Claim \ref{claim 3.4}, there is either a $(v_2, v_7)$-short jump or a $(v_5, v_8)$-short jump. Without loss of generality, assume that there is a $(v_5, v_8)$-short jump. 
However, since $v_2$ is not adjacent to either $v_5$ or $v_8$, by Lemma \ref{lemma 2} there is no short jump with $v_2$ as the end, contradicting $v_2\in S$.
Therefore, either there is a $(v_2, v_5)$-short jump, or there exist both a $(v_1, v_5)$-short jump and a $(v_2, v_6)$-short jump. 

%\noindent \textbf{Claim 3.5.1}
\begin{subclaim}\label{claim 3.5.1}
There is no $(v_2, v_5)$-short jump.
\end{subclaim}
\pf Suppose not.
Since $\{v_7, v_8\}$ is anticomplete to $\{v_2, v_5\}$, by Lemma \ref{lemma 2} there is no short jump with the end $v_7$ or $v_8$. 
Since $d(v_8)\ge 3$, there is a $(v_7, v_8)$-link or $(v_8, v_1)$-link.
If there is a  $(v_7, v_8)$-link, then either there exists a  $v_6-v_7-v_8$ local jump, or there exists a  $v_1-v_8-v_7$ local jump. By Lemma \ref{lemma 4}, both cases contradict  the  $(v_2, v_5)$-short jump. Therefore, there is no  $(v_7, v_8)$-link. Hence, there exists  $(v_8, v_1)$-link. Similarly, there exists a  $(v_6, v_7)$-link, a $(v_4, v_5)$-link, and a $(v_2, v_3)$-link. Furthermore, since $\{v_8, v_1\}$ is not a cut,  the $(v_8, v_1)$-link must connect to $v_2$, $v_5$, or $v_6$.
That is, there exists a $v_2-v_1-v_8$ local jump, or a $v_8-v_1-v_5$ link, or a $v_8-v_1-v_6$ link.

\medskip

\noindent \textbf{Case 3.5.1} There is no $v_2-v_1-v_8$ local jump.
%Next, we prove that there is no $v_2-v_1-v_8$ local jump.

Suppose not. Let $P$ be a $v_2-v_1-v_8$ local jump, 
and let $Q$ be a  $(v_2, v_5)$-short jump. Then $P^*$ and $Q^*$ are not anticomplete; 
Otherwise, $P \cup Q \cup v_5v_6v_7v_8$ would contain an even hole of length greater than 8, a contradiction.
Let $u$ be the neighbor of $v_1$ on $P^*$ closest to $v_8$, and let $v$ be the neighbor of $v_1$ on $P^*$ closest to $v_2$. Partition $P$ into three subpaths:
$P_1 = v_8Pu, P_2 = uPv, P_3 = vPv_2.$
Since there is no  $(v_2, v_8)$-short jump and no $(v_5, v_8)$-short jump, it follows that $P_1^*$ and $Q^*$  are anticomplete.
We now claim that $P_2$ and $Q$ are not anticomplete. Suppose not. 
Then there exists an edge between ${P_3}^*$ and $Q^*$.
Let $w'$ be the vertex on $Q$ closest to $v_5$ that has a neighbor in ${P_3}^*$. Then $w'v_5 \in E$.
Otherwise, $w'v_5 \notin E$.
If $w'$ has exactly one neighbor on $P_3$, denoted by $w_1$, then
$P \cup w_1w' Qv_5 \cup C \setminus v_2v_1v_8$
forms a theta graph. 
Since $|v_8Pw_1| \geq 7$, it follows that
$v_2P_3w_1w'Qv_5v_4v_3v_2$ is an even hole of length $8$.
In this case,
$v_2P_3w_1w'Qv_5v_6v_7v_8v_1v_2$
is an even hole of length greater than $8$, a contradiction.
If $w'$ has at least two neighbors on $P_3$, let $w_1$ be the neighbor closest to $v_8$, and let $w_2$ be the neighbor closest to $v_2$.
Then $v_8Pw_1 \cup w_2Pv_2 \cup C \setminus v_8v_1v_2 \cup w'Qv_5$
forms a theta graph, and since $|v_8Pw_1| \geq 7,$
it follows that $v_2P_3w_2w'Qv_5v_4v_3v_2$
is an even hole of length $8$.
In this case, $v_2P_3w_2w'Qv_5v_6v_7v_8v_1v_2$ 
is an even hole of length greater than $8$, a contradiction.
Therefore, $w'v_5\in E$.
Similarly, $w'$ has at least two nonadjacent neighbors on $P_3$.
%Let $x_1\in N_P(v_1)$  closest to $v_2$. 
Let $x_1\in N_P(w')$  be closest to $v$, 
and let $x_2\in N_P(w')$  be closest to $v_2$. 
Suppose $v_1vPx_1w'x_2Pv_2v_1$ is a hole of length $8$. 
Then one of the following holds:
either $vx_1 \in E$ and $|x_2Pv_2| = 3,$ in which case $v_2Px_2w'v_5v_6v_7v_8v_1v_2$ is a hole of length $10$;
or $|vPx_1|=|x_2Pv_2|=2,$ in which case $ v_2Px_2w'v_5v_4v_3v_2$
is a hole of length $7$;
or $|vPx_1|=3$ and $x_2v_2 \in E,$
in which case $v_2Px_2w'v_5v_4v_3v_2$ is a hole of length $6$.
Both cases yield a contradiction.
Therefore, $v_1vPx_1w'x_2Pv_2v_1$ is an odd hole.
Since the $(v_2, v_5)$-short jump has even length, it follows that
$v_2Px_2$ has even length, with $|v_2P_3x_2| \geq 4,$
and $vP_3x_1$ has odd length.
Since $v_1vP_3x_1w'v_5$ is either odd or has length $4$,  $vx_1\in E.$
Moreover, since $v_8Px_1w'v_5v_6v_7v_8$ is an odd hole, $v_8Pvx_1$
has even length.
Also, since $ v_1vx_1P_3x_2v_2v_1$ has length greater than $8$, it is an odd hole, so $x_1P_3x_2$ has even length.
Thus, $v_8Pv_2$ has even length, a contradiction.
Therefore, $P_2$ and $Q$ are not anticomplete.

Let $w'$ be the vertex on $Q$ closest to $v_5$ that has a neighbor on $P$. Then $w'$ has a neighbor on $P_2$, and by the above argument,
$w'v_5 \in E.$
If $w'$ has exactly one neighbor on $P$, denoted by $x$, where
$x\in V(P_2),$ then $P\cup v_5w'x \cup C \setminus v_8v_1v_2$
forms a theta graph. Since  $|v_8Px|, |v_2Px|\geq 6,$
this theta graph induces three odd holes, which is impossible.
Therefore, $w'$ has at least two neighbors on $P$. 
Let $x_1$ be the neighbor of $w'$ on $P_2$ closest to $u$.
We now prove that $N_P(w') \subseteq N_{P_2}(w').$
Suppose not.
If $w'$ has at least two neighbors on $P_3$, let $x_2$ be the neighbor closest to $v$, and let $x_3$ be the neighbor closest to $v_2$. Then
$|v_2P_3x_3| \geq 4$ and is even.
If $vPx_2$ has even length, then $v_1vPx_2w'x_3Pv_2v_1$ is an even hole of length greater than $8$, a contradiction.
Therefore, $vPx_2$ has odd length.
Similarly, $uPx_1$ has odd length.
Since a $(v_1, v_5)$-short jump is either of length $4$ or odd, it follows that $ux_1 \in E,  vx_2\in E.$
Thus, $v_1ux_1w'x_2vv_1$ is a hole of length $4$, a contradiction.
Hence, $w'$ has exactly one neighbor on $P_3$, denoted by $x_2$.
Then $uPx_1, vPx_2$ are either of length $1$ or of even length.
If $|uPx_1| = 1,$ then $|vPx_2|$ is even and $|vPx_2|\geq 4.$
Since $|v_2Px_2|$ is also even and at least $4$, $v_1vPv_2v_1$
is an even hole of length greater than $8$, a contradiction.
Therefore, $|uPx_1|$ is even.
Since $|v_2Px_2|$ is even and at least $4$, $v_1uPx_1w'x_2Pv_2v_1$
is an even hole of length greater than $8$, again a contradiction.
Therefore, in all cases, $N_P(w') \subseteq N_{P_2}(w').$
Let $x_1\in N_{P_2}(w')$ be closest to $u$, and let $x_2\in N_{P_2}(w')$ be closest to $v$. 
Then $ux_1\notin E, vx_2\notin E.$
Suppose not. Without loss of generality, assume $ux_1 \in E.$
Then $v_8Pu$ is a $(v_8, u)$-short jump on $v_1ux_1w'v_5v_6v_7v_8v_1$,
contradicting  Claim \ref{claim 3.4}.
Therefore, $ux_1 \notin E, vx_2 \notin E.$
Since the $(v_1, v_5)$-short jump is type-o,
$uPx_1$ and $x_2Pv$ both have even length.
Thus, $v_1uPx_1w'x_2Pvv_1$ is an even hole, so it must have length $8$. Hence, $|uPx_1| = |x_2Pv| = 2.$
Let $x$ be the internal vertex of $uPx_1$, and let $x'$ be the internal vertex of $x_2Pv$.
Since $Q$ is type-o, $v_2Qw'$ has odd length and length at least $5$.
Moreover, neither $\{u, x, x_1\}$ nor $\{x_2, x', v\}$ is anticomplete to
$(v_2Qw')^*.$ 
Suppose not. Without loss of generality, assume that $\{u, x, x_1\}$ is anticomplete to $(v_2Qw')^*.$ 
Then $v_1uxx_1w'Qv_2v_1$ is an even hole of length greater than $8$, a contradiction.
Let $y\in V(Q)$ be closest to $v_2$ that has a neighbor in ${u, x, x_1, x_2, x', v}.$
Since $v_1uxx_1w'x_2x'vv_1$ is an $8$-hole, $y$ has exactly one neighbor in $\{u, x, x_1, x_2, x', v\}.$
We now prove that $yv \in E.$
Suppose not.
If $yu \in E,$ then since $|v_8Pu|\geq 6,$
$v_8Puxx_1w'v_5v_6v_7v_8$ is an odd hole, so $|v_8Pu|$ is even.
Moreover, $v_8PuyQv_2v_3\cdots v_8$ is also an odd hole, which implies that
$v_2Qy$ has even length.
Thus, $v_2Qyuxx_1w'v_5v_4v_3v_2$ is an even hole of length greater than $8$, a contradiction.
If $yx \in E,$ then $v_1v_2Qyx$ is a $(v_1, x)$-short jump on
$v_1uxx_1w'x_2x'vv_1$, contradicting  Claim \ref{claim 3.4}.
If $yx_1\in E$, then since $v_8Px_1$ has even length and
$v_8Px_1yQv_2$ has odd length, $yQv_2$ has even length.
In this case, both $v_2Qyx_1w'v_5v_4v_3v_2$ and $v_2Qyx_1w'v_5v_6v_7v_8v_1v_2$ are even holes of different lengths, a contradiction.
If $yx_2 \in E$,
then $v_2Qyx_2w'v_5$ is a short jump over $C$, so $v_2Qy$ has odd length.
Moreover, $v_1vx'x_2yQv_2v_1$ is an even hole, so $|v_2Qy| = 3.$
Then $v_1uxx_1w'x_2yQv_2v_1$ is a $10$-hole, a contradiction.
If $yx' \in E$, then $v_1v_2Qyx'$ is a $(v_1, x')$-short jump over
$v_1uxx_1w'x_2x'vv_1,$ again contradicting  Claim \ref{claim 3.4}.
Therefore, $yv \in E.$
If $y \notin V(P_3)$, then $v_2Qyvx'x_2w'v_5v_4v_3v_2$ has length greater than $8$, so it is an odd hole. 
Hence, $v_2Qy$ has odd length.
Thus, $v_1vyQv_2v_1$ is an even hole, and therefore it must be an $8$-hole.
Then $v_2v_3v_4v_5w'x_2x'v$ is a $(v_2, v)$-short jump over
$v_1vyQv_2v_1$, contradicting  Claim \ref{claim 3.4}.
Therefore, $y \in V(P_3)$.
Similarly, $v_2Qy=v_2P_3y,$ and this path has odd length.
Then again, $v_1vyQv_2v_1$ is an even hole, so it must be an $8$-hole.
Thus, $v_2v_3v_4v_5w'x_2x'v$ is a $(v_2, v)$-short jump over $v_1vyQv_2v_1,$ contradicting  Claim \ref{claim 3.4}.
Therefore, no $v_2-v_1-v_8$ local jump exists.
By symmetry, there is also no $v_3-v_2-v_1$ local jump,
$v_7-v_6-v_5$ local jump, or
$v_4-v_5-v_6$ local jump.
\qed

\medskip

\noindent \textbf{Case 3.5.2} There is no $v_8-v_1-v_6$ link.
%Next, we prove that there is no $v_8-v_1-v_6$ link.

Suppose not.
Let $w_1$ be the neighbor of $v_1$ on the $v_8-v_1-v_6$ link that is closest to $v_6$. Let $Q$ be the subpath of this link from $v_8$ to $w_1$.
Let $w$ be the neighbor of $v_6$ on the $v_8-v_1-v_6$ link, and let
$P_1$ be the subpath from $w_1$ to $w$ along this link.
If there also exists a $v_4-v_5-v_2$ link, 
let $w_2$ be the neighbor of $v_5$ on this link that is closest to $v_2$. Let $Q'$ be the subpath of the $v_4-v_5-v_2$ link from $v_4$ to $w_2$.
Let $w'$ be the neighbor of $v_2$ on this link, and let $P_2$
be the subpath from $w_2$ to $w'$ along the $v_4-v_5-v_2$ link.
Since there is no $v_8-v_1-v_2$ local jump and no $(v_2, v_8)$-short jump, 
$Q\cup P_1\backslash \{w\}$ is anticomplete to $P_2\backslash \{w_2\}$.  
If $w_2$ has a neighbor on $P_1$, then $w_2$ can only have neighbors on ${P_1}^*$, denoted $w''$.  
Then $Q' \cup w''P_1w$ either contains a $v_4-v_5-v_6$ local jump or a $(v_4, v_6)$-short jump, a contradiction.  
Thus $P_1\backslash \{w\}$ is anticomplete to $P_2$.  
Since $P_1$ and $P_2$ cannot be anticomplete, $w$ has a neighbor in $P_2 \backslash \{w'\}$.  
Then $Q'\cup {P_2}^* \cup \{w\}$ contains a $v_4-v_5-v_6$ local jump or a $(v_4, v_6)$-short jump, a contradiction.
So there does not exist a $v_4-v_5-v_2$ link.
If there is simultaneously a $v_3-v_2-v_5$ link,
denote by $w_2$ the neighbor of $v_2$ on the $v_3-v_2-v_5$ link that is closest to $v_5$, 
denote by $Q'$ the subpath from $v_3$ to $w_2$ on the $v_3-v_2-v_5$ link, 
denote by $w'$ the neighbor of $v_5$ on the $v_3-v_2-v_5$ link, 
and denote by $P_2$ the subpath from $w_2$ to $w'$ on the $v_3-v_2-v_5$ link. 
Then $Q\cup {P_1}^*$ is anticomplete to $Q'\cup {P_2}^*$;
Otherwise there is a $(v_8, v_3)$-short jump, 
a $v_8$-$v_1$-$v_2$ local jump, a $(v_8, v_2)$-short jump, 
or a $(v_1, v_3)$-short jump, a contradiction. 
Since $P_1$ and $P_2$ cannot be anticomplete, either $w$ has a neighbor in $P_2$ or $w'$ has a neighbor in $P_1$. 
Without loss of generality, assume $w$ has a neighbor in $P_2$.
Let $w^*\in N_{P_2}(w)$ that is closest to $w_2$. 
By the parity of the $(v_2, v_6)$-short jump, 
either $w_2w^* \in E$ or $|w_2P_2w^*|$ is even.
Since $v_1w_1P_1ww^*P_2w_2v_2v_1$ has length greater than $8$, 
we have $w_2w^* \in E$. 
Because $w_2Q'v_3v_4v_5w'P_2w_2$ has length greater than $8$, $|w_2Q'v_3|$ is even. 
If $w$ has no neighbor in $Q'$, then $w_2Q'v_3v_4v_5v_6ww^*w_2$ is an even cycle of length greater than $8$. 
Hence $w$ has a neighbor in $Q'$; 
let $x\in N_{Q'}(w)$ be closest to $v_2$, 
and let $y\in  N_{Q'}(v_2)$ be closest to $x$. 
Then $v_2yQ'xwv_6$ is a $(v_2, v_6)$-short jump, 
so either $xy \in E$ or $|xP_2y|$ is even. 
If $xy \in E$, then either $w_2xww^*w_2$ is a $4$-hole or $v_2yxww^*w_2v_2$ is a $6$-hole, a contradiction. 
If $|xP_2y|$ is even, then $v_1w_1P_1wxQ'yv_2v_1$ is an even hole of length greater than $10$, a contradiction. 
Therefore there is no $v_3-v_2-v_5$ link; only a $v_3-v_2-v_6$ link and a $v_4-v_5-v_1$ link are possible.
Let $w_1$ be the neighbor of $v_2$ on the $v_3-v_2-v_6$ link that is closest to $v_6$, 
let $Q$ be the subpath from $v_3$ to $w_1$ on the $v_3-v_2-v_6$ link, let $w$ be the neighbor of $v_6$ on the $v_3-v_2-v_6$ link, 
and let $P$ be the subpath from $w_1$ to $w$ on the $v_3-v_2-v_6$ link. 
Let $w_2$ be the neighbor of $v_5$ on the $v_4-v_5-v_1$ link that is closest to $v_1$, 
let $Q'$ be the subpath from $v_4$ to $w_2$ on the $v_4-v_5-v_1$ link, let $w'$ be the neighbor of $v_1$ on the $v_4-v_5-v_1$ link, and let $P'$ be the subpath from $w_2$ to $w'$ on the $v_4-v_5-v_1$ link. 
Then $Q$ and $Q'$ are anticomplete; Otherwise  there would be a $(v_3, v_4)$-link, a $(v_2, v_4)$-short jump, a $(v_3, v_5)$-short jump, a $v_3$-$v_2$-$v_5$ link, or a $v_4-v_5-v_2$ link, a contradiction. 
Moreover, $P'$ is anticomplete to $Q \cup P$; Otherwise  there would be a cycle of length less than $8$, a $(v_3, v_1)$-short jump, or a $v_3-v_2-v_1$ local jump, a contradiction. 
Similarly, $P$ is anticomplete to $Q' \cup P'$. Since $v_1w'P'w_2Q'v_4v_3v_2v_1$ is an odd hole, $v_1w'P'w_2Q'v_4$ has even length. 
Similarly, $v_6wPw_1Qv_3$ has even length. 
Then $v_1v_8v_7v_6wPw_1Qv_3v_4Q'w_2P'w'v_1$ is an even hole of length greater than $8$, a contradiction. 
Therefore there is no $v_8-v_1-v_2$ link. \qed

\medskip

Similarly, there is no $v_3-v_2-v_5$ link and no $v_4-v_5-v_2$ link. 
Thus only the $v_3-v_2-v_6$ link and the $v_4-v_5-v_1$ link can exist. Similar to the above discussion, this leads to a contradiction.
Therefore, Subclaim \ref{claim 3.5.1} holds.
%that is, there is no $(v_2, v_5)$-short jump. 
%To prove Claim \ref{claim 3.5}, it suffices to show that there is no $(v_1, v_5)$-short jump and no $(v_2, v_6)$-short jump.
%This part of the discussion is placed in the following claim.
\qed  

%\noindent \textbf{Claim 3.5.2} 
\begin{subclaim}\label{Claim 3.5.2}
There is no $(v_1, v_5)$-short jump and no $(v_2, v_6)$-short jump.
\end{subclaim}
\pf Suppose not.
Since $d(v_8)\ge 3$,  there is either a $(v_7, v_8)$-link or a $(v_8, v_1)$-link. 
If there is a $(v_7, v_8)$-link, then there is a $v_6-v_7-v_8$ local jump or a $v_7-v_8-v_1$ local jump. 
If there is a $v_7-v_8-v_1$ local jump, then by Lemma \ref{lemma 3}, the $(v_2, v_6)$-short jump is of type-e, and in this case there is no $v_6-v_7-v_8$ local jump. 
If there is only a $(v_1, v_8)$-link, then it is a $v_8-v_1-v_2$ local jump, 
or a $v_8-v_1-v_5$ link, or a $v_8-v_1-v_6$ link. 
If there is a $v_7-v_8-v_1$ local jump, 
then there is no $(v_3, v_4)$-link; Otherwise , either there is a $v_2-v_3-v_4$ local jump, and by Lemma \ref{lemma 3}, 
there is a short jump with end $v_3$ or $v_8$, or there is a $v_3-v_4-v_5$ local jump, and by Lemma \ref{lemma 3},  there is a short jump with end $v_4$ or $v_8$, a contradiction. 
Hence either there is no $(v_7, v_8)$-link, or there is no $(v_3, v_4)$-link. 
We may assume there is no $(v_7, v_8)$-link, so there exist both a $(v_1, v_8)$-link and a $(v_6, v_7)$-link.

\medskip

\noindent \textbf{Case 3.5.3} There is no $v_2-v_1-v_8$ local jump. 

%We now prove that there is no $v_2$-$v_1$-$v_8$ local jump. 
Suppose not. 
Denote the $v_2-v_1-v_8$ local jump by $Q$, 
and let $x$ be the neighbor of $v_1$ on $Q$ that is closest to $v_8$.
Assume that simultaneously there is a $v_5-v_6-v_7$ local jump.  
Denote the  $v_5-v_6-v_7$ local jump by $P$, 
and let $w'_1$ be the neighbor of $v_6$ on $P$ that is closest to $v_7$, and $w'_2$ the neighbor of $v_6$ on $P$ that is closest to $v_5$. 
Since $Q$ is odd and $P$ is odd, $Q^*$ and $P^*$ are not anticomplete; Otherwise  $P \cup Q \cup v_7v_8 \cup v_2v_3v_4v_5$ would be an even hole of length greater than $8$.
Since there is no $(v_7, v_8)$-link and no short jump with end $v_7$ or $v_8$, 
the only possibility is that $xQv_2$ and $w'_1Pv_5$ are not anticomplete. We now show that any edge between $xQv_2$ and $w'_1Pv_5$ must have an end that is a neighbor of $v_2$ on $Q$ or a neighbor of $v_5$ on $P$.
Suppose otherwise. 
Then $Q \cup \{v_1\}$ contains a path with end $v_1$, denoted $P'$, whose other end is denoted $w$, such that $N_P(w) \neq \emptyset$ and $(P')^*$ is anticomplete to $P$. 
Assume $v_1w \notin E$. If $|N_P(w)| = 1$, then $v_5Pv_7v_8v_1v_2v_3v_4v_5 \cup v_1P'w$ is a theta graph, and none of the three holes in this theta graph is an even hole, a contradiction. 
If $|N_P(w)| \ge 2$, then $v_5Pw_2ww_1Pv_7v_8v_1v_2v_3v_4v_5 \cup v_1P'w$ is a theta graph, where $w_1\in N_P(w)$  is closest to $v_5$, $w_2\in N_P(w)$ is closest to $v_7$, and none of the three holes in this theta graph is an even hole, a contradiction. 
Therefore, $v_1w \in E$.
If $w$ has only one neighbor on $P$,  denoted by $w'$. 
Since $v_5Pv_7v_8v_1v_2v_3v_4v_5 \cup v_1w$ induces a theta graph, it must contain an even hole, so $v_1ww'Pv_5v_4v_3v_2v_1$ is an $8$-hole and $d_P(w', v_5) = 2$. 
Then $w'Pw'_2v_6v_5$ is a $(w', v_5)$-short jump over $v_1ww'Pv_5v_4v_3v_2v_1$, contradicting  Claim \ref{claim 3.4}. 
Hence $w$ has at least two neighbors on $P$.
Let $w_1\in N_P(w)$ be closest to $w'_1$, and let $w_2\in N_P(w)$ be closest to $v_5$. 
We now prove that $N_P(w) \not\subseteq P_2$. 
Suppose otherwise, $w_2\in V(P_1)$ is closest to $w'_2$. 
Since the $(v_1, v_6)$-short jump has even length, 
$w'_1Pw_1$ and $w'_2Pw_2$ have odd length at least $3$. 
Then $ww_1Pw'_1v_6w'_2Pw_2w$ is an even hole of length greater than $8$, a contradiction. 
Hence $N_{P_3}(w) \neq \emptyset$. 
Suppose $|N_{P_3}(w)|=1$, then $w_2 \in V(P_3)$. 
Then $v_1ww_2Pv_5$ is a $(v_1, v_5)$-short jump over $C$, 
so either $|w_2Pv_5| = 2$ or $|w_2Pv_5|$ is odd. 
If $|w_2Pv_5| = 2$, then $w_2Pw'_2v_6v_5$ is a short jump over $v_1ww_2Pv_5v_4v_3v_2v_1$, contradicting  Claim \ref{claim 3.4}. 
Thus $|w_2Pv_5|$ is odd. By the parity of the $(v_1, v_6)$-short jump, $|w_1Pw'_1|$ is odd of length at least $3$, and then $ww_1Pw'_1v_6v_5Pw_2w$ is an even hole of length greater than $8$, a contradiction. 
Therefore $|N_{P_3}(w)| \ge 2$.
Let $w_3 \in N_{P_3}(w)$ be closest to $w'_2$. 
Then by the parity of the $(v_1, v_6)$-short jump, both $|w_1Pw'_1|$ and $|w_3Pw'_2|$ are odd and at least $3$. 
Then $ww_1Pw'_1v_6w'_2Pw_3w$ is an even hole of length greater than $8$, a contradiction.
Thus, every edge between $xQv_2$ and $w'_1Pv_5$ must have an end that is either a neighbor of $v_2$ on $Q$ or a neighbor of $v_5$ on $P$. 
Let $w \in N_Q(v_2)$ and assume without loss of generality that $w$ has a neighbor in $P$. 
If $w$ has exactly one neighbor in $P$,  denoted by $w'$. 
Since $v_5Pv_7v_8v_1\cdots v_5 \cup v_2ww'$ is a theta graph and must contain an even hole, the only possibility is that $v_2ww'Pv_5v_4v_3v_2$ is an $8$-hole, but then $v_2ww'Pv_5v_6\cdots v_8v_1v_2$ is a $10$-hole, a contradiction.
Hence $w$ has at least two neighbors in $P$; 
let $w_1$ be the neighbor closest to $w'_1$ and $w_2$ be the neighbor closest to $w'_2$.
Since $wQv_8v_7v_6w'_1Pw_1w$ has length greater than $8$, we have $w'_1w_1 \in E$. 
Similarly, $w'_2w_2 \in E$. 
Then $ww_1w'_1v_6w'_2w_2w$ is a $4$-hole, a contradiction. 
Therefore there is no $v_5-v_6-v_7$ local jump.

Suppose there is a $v_7-v_6-v_2$ link, and denote it by $P$. 
Choose $P$ such that $|Q \cup P|$ is minimized.
Then $P^*$ and $Q^*$ are not anticomplete; 
Otherwise there would be either a $(v_7, v_8)$-link or an even hole of length greater than $8$. 
Since there is no short jump with ends $v_7$ or $v_8$, 
$P \cup Q$ induces a path, denoted by $P'$, with one end $v_6$ and the other end $w$, where $w$ is the unique vertex of $V(P')$ that has a neighbor in $Q$. 
Clearly $N_Q(w) \subseteq xQv_2$. 
Let $w'$ be the neighbor of $w$ in $Q$ closest to $v_2$. 
If $v_6w \notin E$, then since $Q \cup v_2v_3\cdots v_8 \cup v_6P'w$ contains a theta graph, 
and since theta graph contains an even hole, 
we have $w'v_2 \in E$ and $|v_6P'w| = 2$. 
If $|N_Q(w)| = 1$, then $Q \cup \{v_1\}$ contains a $(v_1, w')$-short jump over $v_1v_8v_7v_6P'ww'v_2v_1$, contradicting  Claim \ref{claim 3.4}. 
Hence $|N_Q(w)| \ge 2$; then $Q \cup \{v_1, w\}$ contains a $(v_2, w)$-short jump over $v_2v_3v_4v_5v_6P'ww'v_2$, contradicting  Claim \ref{claim 3.4}. 
Thus $v_6w \in E$. 
If $|N_Q(w)| = 1$, similarly $d_Q(w', v_2) = 2$, 
and $Q \cup \{v_1\}$ contains a $(v_2, w')$-short jump over $v_2v_3v_4v_5v_6ww'Qv_2$, contradicting  Claim \ref{claim 3.4}. 
Hence $|N_Q(w)| \ge 2$. 
Let $w''\in N_P(w')$ be closest to $v_7$.
From the above, $w''v_6 \in E$, and $w''$ could be $w$. 
Since $w''w'Qv_2v_1v_8v_7Pw''$ has length greater than $8$, it is an odd hole, 
and $v_7Pw''$ has odd length greater than $7$. 
Let $x'\in N_Q(w'')$ be closest to $x$. Since $v_8Qx'w''Pv_7v_8$ has length greater than $8$, it is an odd hole, so $v_8Qx'$ has even length greater than $4$. 
Then $v_8Qx'w''v_6v_7v_8$ is an even hole of length greater than $8$, a contradiction. 
Therefore there is no $v_7-v_6-v_2$ link.

Suppose there is a $v_7-v_6-v_1$ link, and denote it by $P'$.
Let $w\in N_{P'}(v_6)$  be closest to $v_1$, and let $w'\in N_{P'}(v_1)$. Denote $v_6wP'v_1$ by $P$. 
By Claim \ref{claim 3.5.1}, there is no $(v_2, v_5)$-short jump.
Assume the $(v_1, v_5)$-short jump is type-e; then every $(v_2, v_6)$-short jump has odd length. 
Take a $(v_2, v_6)$-short jump and denote it by $Q$. 
Since there is no $v_7-v_6-v_2$ link, $V(P) \cap V(Q) = \emptyset$. Moreover, $P^*$ and $Q^*$ cannot be anticomplete; 
Otherwise $P \cup Q \cup v_1v_2$ would be an even hole of length greater than $8$. 
Since there is no $v_7-v_6-v_2$ link, only $w'$ of $V(P^*)$ has neighbors in $Q$. 
Let $w_1\in N_Q(w')$ be closest to $v_2$. 
Since $w'w_1Qv_2v_3\cdots v_6Pw'$ has length greater than $8$, it is an odd hole, so $v_2Qw_1$ has odd length. 
Then $v_2Qw_1w'v_1v_2$ is an even hole, which must be an $8$-hole. 
Then $w'Pv_6v_5\cdots v_2$ is a $(w', v_2)$-short jump over $v_2Qw_1w'v_1v_2$, contradicting  Claim \ref{claim 3.4}. 
Hence the $(v_1, v_5)$-short jump is type-o.
Let $Q$ be a $(v_1, v_5)$-short jump and denote by $w_1$ the neighbor of $v_1$ on $Q$. 
If $P^*$ and $Q^*$ are anticomplete, then $P \cup Q \cup v_5v_6$ yields an even hole of length greater than $8$, a contradiction. 
Since there is no $v_5-v_6-v_7$ local jump, either $w'$ has a neighbor in $Q$ or $w$ has a neighbor in $P$; 
in either case, $P \cup Q \cup v_5v_6$ contains an even hole of length greater than $8$, a contradiction.
Hence there is no $v_7-v_6-v_1$ link. 
In summary, there is no $v_2-v_1-v_8$ local jump.
Similarly, there is no $v_1-v_2-v_3$ local jump, no $v_4-v_5-v_6$ local jump, and no $v_5-v_6-v_7$ local jump.

\medskip

\noindent \textbf{Case 3.5.4} There is no $v_8-v_1-v_5$ link.

%We now prove that there is no $v_8$-$v_1$-$v_5$ link.
Suppose otherwise, and denote the $v_8-v_1-v_5$ link by $P'$. 
Let $x\in N_{P'}(v_1)$ be the neighbor of $v_1$ on $P'$ closest to $v_8$, and $w$ the neighbor farthest from $v_8$. 
Denote $v_8P'x$ by $P_1$, $xP'w$ by $P_2$, and $v_1wP'v_5$ by $P$. 
Assume simultaneously there is a $v_7-v_6-v_2$ link,  denoted by $Q'$. 
Let $x'$ be the neighbor of $v_6$ on $Q'$ closest to $v_7$, and $w'$ the neighbor farthest from $v_8$. 
Denote $v_7Q'x'$ by $Q_1$, $x'Q'w'$ by $Q_2$, and $v_6w'Q'v_2$ by $Q$.
By  Lemma \ref{c4.1}, we may assume that all $(v_1, v_5)$-short jumps are of type-o, that is, they have odd length. 
We now show that $w$ is anticomplete to $Q'$. 
Suppose otherwise. Since there is no short jump with end $v_7$, 
either $w$ is not anticomplete to $Q_2^*$ or $w$ is not anticomplete to $Q^*$.
Moreover, $w$ is anticomplete to $Q^*$; 
Otherwise there would be a $v_8-v_1-v_2$ local jump. 
Similarly, $P_1 \cup P_2$ is anticomplete to $Q^* \setminus \{w'\}$, $w'$ is anticomplete to $P^*$, and $Q_1 \cup Q_2$ is anticomplete to $P^* \setminus \{w\}$.
Hence $w$ is not anticomplete to $Q_2^*$.
Take the neighbor $w''$ of $w$ in $Q_2$ that is closest to $w'$. 
Then $ww''Q_2w_1v_6v_5Pw$ is an even hole of length greater than $8$, where $w_1$ is the neighbor of $v_6$ on $w''Q_2w'$ closest to $w''$, a contradiction. 
Thus $P^*$ is anticomplete to $Q'^*$ and $V(P) \cap V(Q) = \emptyset$. 
If $P'^*$ is anticomplete to $Q'^*$, then since $|P'|$ and $|Q'|$ are both even, $P' \cup Q' \cup v_7v_8 \cup v_2v_3v_4v_5$ would be an even hole of length greater than $8$, a contradiction.
Hence $P'^*$ and $Q'^*$ are not anticomplete, and $Q_2$ is not anticomplete to $P_2$. 
Let $w''$ be the vertex in $Q_2$ closest to $w'$ that has a neighbor in $P_2$. 
Then there is a path through $w''$ connecting $v_1$ and $v_6$ whose internal vertices lie in $V(P_2) \cup V(Q_2)$. 
Together with $P \cup v_5v_6$, this yields an even hole of length greater than $8$, a contradiction.
Therefore, there is no $v_7-v_6-v_2$ link.

Suppose there is a $v_7-v_6-v_1$ link, denoted by $Q$, and let $x'$ be the neighbor of $v_6$ on $Q$ closest to $v_1$, and denote $v_6x'Qv_1$ by $Q'$; then $Q'$ has even length. 
If there exists a $(v_1, v_5)$-short jump of type-o,  denoted by $Q_1$, then $Q_1$ has odd length. 
In this case, either $Q' \cup Q_1 \cup v_5v_6$ contains an even hole of length greater than $8$, or there is a $v_5-v_6-v_7$ local jump, a contradiction. 
Hence every $(v_1, v_5)$-short jump is type-e. 
Consequently, every $(v_2, v_6)$-short jump is type-o. 
Take a $(v_2, v_6)$-short jump and denote it by $Q_1$. 
Then $V(Q'^*) \cap V(Q_1^*) = \emptyset$; Otherwise  there would be a $v_7-v_6-v_2$ link, a contradiction. 
Since $Q'$ has even length and $Q_1$ has odd length, $Q'^*$ and $Q_1^*$ are not anticomplete; Otherwise  $Q' \cup Q_1 \cup v_1v_2$ would be an even hole of length greater than $8$, a contradiction. 
Let $w'$ be the neighbor of $v_1$ on $Q'$. 
Then only $w'$ on $Q$ has neighbors in $Q_1$; Otherwise there would be a $v_7-v_6-v_2$ link, a contradiction. 
Let $w_1$ be the neighbor of $w'$ in $Q_1$ closest to $v_2$. Since $ww_1Q_1v_2v_3v_4v_5v_6Q'w$ has length greater than $8$, it is an odd hole, so $w_1Q_1v_2$ has odd length. 
Then $v_1ww_1Q_1v_2v_1$ has even length, so it is an $8$-hole. 
Then $wQ'v_6v_5v_4v_3v_2$ is a $(w, v_2)$-short jump on $v_1ww_1Q_1v_2v_1$, contradicting  Claim \ref{claim 3.4}. 
Hence, there is no $v_7-v_6-v_1$ link. 
In summary, there is no $v_8-v_1-v_5$ link. 

If there were a $v_8-v_1-v_6$ link, a similar argument for the case with a $v_7-v_6-v_1$ link would lead to a contradiction. 
Thus there is no $(v_1, v_5)$-short jump and no $(v_2, v_6)$-short jump. 

The proof about $d_C(s,t)\neq 4$ is similar to the proof of Subclaim \ref{Claim 3.5.2}, where there exists a $(s,t)$-short jump over $C$.
Hence, Theorem \ref{theorem 1} is true.
\qed

%Since $\{v_2, v_3\}$ is anticomplete to $\{v_5, v_8\}$, by Lemma \ref{lemma 2} there is no short jump with end $v_2$ or $v_3$. Similarly, because $v_4$ is not adjacent to $v_1, v_6$, there is no short jump with end $v_4$. Then $\{v_2, v_3, v_4\}\cap S=\emptyset$, contradicting Lemma \ref{lemma 4}.

\end{document}